\documentclass[a4paper]{article}
\usepackage{amsthm,amsfonts,amsmath,amssymb}
\usepackage[cp1251]{inputenc}
\usepackage[english]{babel}
\usepackage[final]{graphicx}
\usepackage[table]{xcolor}
\usepackage{setspace}
\usepackage[12pt]{extsizes}
\begin{document}
\newtheorem{thm}{Theorem}
\newtheorem{lem}{Lemma}
\newtheorem{prop}{Proposition}
\newtheorem{cor}{Corollary}
\newtheorem{con}{Conjecture}

\author{A. A. Taranenko}
\title{Positiveness of the permanent of 4-dimensional polystochastic matrices of order 4 \thanks{Sobolev Institute of Mathematics, Novosibirsk, Russia. Email: taa@math.nsc.ru. This work was funded by the Russian Science Foundation under grant 18-11-00136. The author is a Young Russian Mathematics award winner and would like to thank its sponsors and jury.}
}
\date{November 8, 2018}

\maketitle

\begin{abstract}
A nonnegative multidimensional matrix is called polystochastic if the sum of its entries over each line is equal to $1$. In this paper we overview known results on positiveness of the permanent of polystochastic matrices and prove that the permanent of every $4$-dimensional polystochastic matrix of order $4$ is greater than zero.


\end{abstract}

\section{Introduction}

A \textit{$d$-dimensional matrix $A$ of order $n$} is an array $(a_\alpha)_{\alpha \in I^d_n}$, $a_\alpha \in\mathbb R$, where the set of indices $I_n^d= \left\{ (\alpha_1, \ldots , \alpha_d):\alpha_i \in \left\{0,\ldots,n-1 \right\}\right\}$. Given $k\in \left\{0,\ldots,d\right\}$, a \textit{$k$-dimensional plane} in $A$ is a submatrix obtained by fixing $d-k$ indices and letting the other $k$ indices vary from 1 to $n$.   A 1-dimensional plane is said to be a \textit{line}, and a $(d-1)$-dimensional plane is a \textit{hyperplane}.

A matrix $A$ is called a \textit{$(0,1)$-matrix} if all its entries are equal to 0 or 1, and $A$ is a \textit{nonnegative} matrix if for all $\alpha \in I_n^d$ we have $a_\alpha \geq 0$. A nonnegative matrix is \textit{polystochastic} if for each line $l$ it holds $\sum\limits_{\alpha \in l} a_\alpha = 1$. 2-dimensional polystochastic matrices are known as \textit{doubly stochastic} matrices.

A \textit{partial diagonal $p$ of length $m$} in a $d$-dimensional matrix $A$ of order $n$ is a set $\{\alpha^1, \ldots, \alpha^m\}$ of $m$ indices  such that each pair of indices $\alpha^i$ and $\alpha^j$ is distinct in all components. A partial diagonal $p$ is  \textit{positive} in a matrix $A$ if  all entries of $A$ with indices from $p$ are greater than zero.

A \textit{diagonal} in a $d$-dimensional matrix  $A$ of order $n$ is a partial diagonal of length $n$ (the maximal possible length). Denote by $D(A)$  the set of all diagonals in $A$. The \textit{permanent} of a multidimensional matrix $A$ is
$${\rm per} A = \sum\limits_{p \in D(A)} \prod\limits_{\alpha \in p} a_{\alpha}.$$

The permanent of polystochastic matrices is applied for counting transversals in latin squares and hypercubes. A \textit{$d$-dimensional latin hypercube $Q$ of order $n$} is a multidimensional matrix filled by $n$ symbols so that each line contains all different symbols. $2$-dimensional latin hypercubes are usually called \textit{latin squares}. Two latin hypercubes are said to be \textit{equivalent} if one can be put to another by permutations of hyperplanes and by permutations of symbols. A \textit{transversal} in a latin hypercube $Q$ is a diagonal containing all $n$ symbols.

There is a one-to-one correspondence between $d$-dimensional latin hypercubes $Q$ of order $n$ and $(d+1)$-dimensional polystochastic $(0,1)$-matrices $A$ of order $n$: an entry $q_{\alpha_1, \ldots, \alpha_d}$ of $Q$ equals $\alpha_{d+1}$ if and only if an entry $a_{\alpha_1, \ldots, \alpha_{d+1}}$ of $A$ equals $1$. The number of transversals in a latin hypercube $Q$ coincides with the permanent of the corresponding polystochastic matrix $A$. For the first time this correspondence was observed in~\cite{jurkat}.

The main aim of this paper is to put together all recent results on positiveness of the permanent of polystochastic matrices and prove that the permanent of all polystochastic matrices of order and dimension $4$ is positive.

\section{History and motivation}

We start our overview with the well-known Birkhoff theorem stating that every doubly stochastic matrix not only has a positive permanent but can be decomposed into a convex combination of permutation matrices. 
\begin{thm}[Birkhoff]
Let $A$ be a doubly stochastic matrix of order $n$. Then ${\rm per} A > 0$ and moreover
$A = \sum \limits_{i=1}^k \theta_i P_i,$
where $P_1, \ldots, P_k$ are permutation matrices,  $\theta_1, \ldots, \theta_k$ are nonnegative, and $\sum \limits_{i=1}^k \theta_i = 1.$ 
\end{thm}

Meanwhile, for dimensions $d$ greater than $2$ there exist $d$-dimensional po\-ly\-sto\-chas\-tic matrices with zero permanent. The simplest example is a $3$-di\-men\-si\-onal $(0,1)$-matrix corresponding to the Cayley table of a group $\mathbb{Z}_n$ of even order $n$. The fact that the Cayley tables of such groups have no transversals was proved by Euler~\cite{euler}. For latin hypercubes this observation was generalized by Wanless that gives us the following construction of polystochastic matrices $Z_n^d$ with a zero permanent.

\begin{prop}[Wanless,~\cite{wanless}]
Let $Z^{d+1}_{n}$ be the $(d+1)$-dimensional polystochastic $(0,1)$-matrix of order $n$ such that $z_\alpha = 1$ if and only if $\alpha_1 + \ldots + \alpha_{d+1} \equiv 0 \mod n$ and let $Q_n^d$ be a $d$-dimensional latin hypercube corresponding to this matrix. 
If $d$ and $n$ are even then the latin hypercube $Q_n^d$ has no transversals.
\end{prop}

There are no known examples of latin squares of odd order with no transversals, and in 1967 Ryser conjectured the following.

\begin{con}[Ryser,~\cite{ryser}]
All latin squares of odd order have a transversal.
\end{con}

This conjecture is related to the conjecture of Stein~\cite{stein} and Brualdi~\cite{brualdi} claiming that every latin square of order $n$ has a partial transversal of length $n-1$. 
Both conjectures have attracted a lot of attention and motivated a number of researchers in last years (see, e.g., the recent works~\cite{aharoni,keevtrans,pokrovskiy} and survey~\cite{wanless} for some history). The Ryser's conjecture is equivalent to that all 3-dimensional polystochastic $(0,1)$-matrices have a positive permanent.

In~\cite{sun} Sun proved that latin hypercubes corresponding to even-dimensional matrices  $Z^d_n$ have a transversal, and so all such matrices have a positive permanent. He also proposed that all $4$-dimensional polystochastic $(0,1)$-matrices have a permanent greater than zero.

\begin{con}[Sun,~\cite{sun}]
Every $3$-dimensional latin hypercube has a transversal.
\end{con}

In~\cite{wantrans} McKay, McLeod, and Wanless and in~\cite{cencus} McKay and Wanless looked through all latin squares and latin hypercubes of small orders and dimensions. Counting transversals in all of them yields the following.

\begin{prop}
\begin{itemize}
\item Every latin square of odd order $n \leq 9$ has a transversal. 
\item Every $3$-dimensional latin hypercube of order $n \leq 6$ has a transversal. 
\item Except for latin hypercubes corresponding to matrices $Z^5_2$ and $Z^5_4$, all $4$-dimensional latin hypercubes of order $n \leq 5$ have a transversal. 
\item Every $5$-dimensional latin hypercube of order $n \leq 5$ has a transversal. 
\end{itemize}
\end{prop}

On the basis of these results, Wanless put forward the following conjecture.
\begin{con}[Wanless,~\cite{wanless}] \label{hyp01}
Every latin hypercube of odd order or odd dimension has a transversal.
\end{con}

This conjecture generalizes the Ryser's and the Sun's conjectures, and the following conjecture, in turn, generalizes all of them.

\begin{con}[Taranenko,~\cite{myobz}] \label{mainhyp}
The permanent of every polystochastic matrix of odd order or even dimension is greater than zero.
\end{con} 

For multidimensional matrices of small orders this conjecture is confirmed by the author.

\begin{thm}
\begin{itemize}
\item Except for matrices $Z_2^d$ of odd dimensions, all polystochastic matrices of order $2$ have a positive permanent. (Taranenko,~\cite{myobz})
\item All polystochastic matrices of order $3$ have a positive permanent. (Taranenko,~\cite{myobz})
\item Except for matrices $Z_4^d$ of odd dimensions, all polystochastic $(0,1)$-matrices of order $4$ have a positive permanent. (Taranenko,~\cite{myquasi})
\end{itemize}
\end{thm}

The main result of the present note is a new supporting case for Conjecture~\ref{mainhyp}.

\begin{thm} \label{poly44}
The permanent of every $4$-dimensional polystochastic matrix of order $4$ is greater than zero. 
\end{thm} 

The following table summarizes all  known results on Conjecture~\ref{mainhyp}.
\begin{center}
\begin{tabular}{|c||c|c|c|c|c|c|c|c|c|c|}
\hline
$n \setminus d$ & 2 & 3 & 4 & 5 & 6 & 7 & 8 & $\ldots$ & $2k$ & $2k+1$ \\
\hline \hline
2 & ~~+~~ & \cellcolor{gray!50} & + & \cellcolor{gray!50} & + & \cellcolor{gray!50} ~~~~~ & + & $\ldots$ & + & \cellcolor{gray!50} \\
\hline
3 & + & + & + & + & + & + & + & $\ldots$ & + & + \\
\hline 
4 & + & \cellcolor{gray!50} & + & \cellcolor{gray!50} & $(0,1)$ & \cellcolor{gray!50} & $(0,1)$ & $\ldots$ & $(0,1)$ & \cellcolor{gray!50} \\
\hline
5 & + & $(0,1)$ & $(0,1)$ & $(0,1)$ & $(0,1)$ &  &  & $\ldots$ &  &  \\
\hline 
6 & + & \cellcolor{gray!50} & $(0,1)$ & \cellcolor{gray!50} &  & \cellcolor{gray!50} &  & $\ldots$ &  & \cellcolor{gray!50} \\
\hline
7 & + & $(0,1)$ &  &  &  &  &  & $\ldots$ &  &  \\
\hline 
8 & + & \cellcolor{gray!50} &  & \cellcolor{gray!50} &  & \cellcolor{gray!50} &  & $\ldots$ &  & \cellcolor{gray!50} \\
\hline
9 & + & $(0,1)$ &  &  &  &  &  & $\ldots$ &  &  \\
\hline 
10 & + & \cellcolor{gray!50} &  & \cellcolor{gray!50} &  & \cellcolor{gray!50} &  & $\ldots$ &  & \cellcolor{gray!50} \\
\hline
11 & + &  &  &  &  &  &  & $\ldots$ &  &  \\
\hline 
$\vdots$ & $\vdots$ & $\vdots$ & $\vdots$ & $\vdots$ & $\vdots$ & $\vdots$ & $\vdots$ & $\ddots$ & $\vdots$ & $\vdots$ \\
\hline
$2m$ & + & \cellcolor{gray!50} &  & \cellcolor{gray!50} &  & \cellcolor{gray!50} &  & $\ldots$ &  & \cellcolor{gray!50} \\
\hline
$2m+1$ & + &  &  &  &  &  &  & $\ldots$ &  &  \\
\hline 
\end{tabular}
\end{center}

\begin{center}
\textbf{Table 1.} 
Gray cells correspond to parameters for which there exist polystochastic matrices with zero permanent, ``$+$'' means that all polystochastic matrices of such dimension and order have a positive permanent, and ``$(0,1)$'' is used for cases when a proof of the conjecture is known only for polystochastic $(0,1)$-matrices. For empty cell parameters Conjecture~\ref{mainhyp} remains completely open. 
\end{center}

\section{Auxiliary lemmas}

A $k \times m$  \textit{row-latin rectangle} $R$ is a table with $k$ rows and $m$ columns filled by $m$ symbols in such a way so that each row contains all $m$ symbols. A \textit{transversal} in the rectangle $R$ is the set of $\min\left\{k,m\right\}$ entries hitting each row, each column and each symbol no more than once. Two row latin rectangles are said to be \textit{equivalent} if one can be put to the other by row, column and symbol permutations.

\begin{lem} \label{rectangle}
Up to equivalence, the row-latin rectangle
$$T = \begin{array} {ccc}
1 & 2 & 3  \\  1 & 2 & 3 \\   2 & 3 & 1  \\  2 & 3 & 1
\end{array}$$
is the unique $4 \times 3$ row-latin rectangle with no transversals. Moreover, if we change any symbol of this rectangle to other one, then we get a (not necessary row-latin) rectangle with a transversal.
\end{lem}

\begin{proof}
Let us list all $4 \times 3$ row-latin rectangles up equivalence:
$$\begin{array} {ccc} \textbf{\underline{1}} & 2 & 3  \\  1 & \textbf{\underline{2}} & 3 \\   1 & 2 & \textbf{\underline{3}}  \\  1 & 2 & 3 \end{array}~~~~~~
\begin{array} {ccc} \textbf{\underline{1}} & 2 & 3  \\  1 & \textbf{\underline{2}} & 3 \\   1 & 2 & \textbf{\underline{3}}  \\  1 & 3 & 2 \end{array}~~~~~~
\begin{array} {ccc} \textbf{\underline{1}} & 2 & 3  \\  1 & \textbf{\underline{2}} & 3 \\   1 & 2 & \textbf{\underline{3}}  \\  2 & 3 & 1 \end{array}~~~~~~
\begin{array} {ccc} 1 & 2 & 3  \\  \textbf{\underline{1}} & 2 & 3 \\   1 & \textbf{\underline{3}} & 2  \\  1 & 3 & \textbf{\underline{2}} \end{array}~~~~~~
\begin{array} {ccc} 1 & \textbf{\underline{2}} & 3  \\  1 & 2 & \textbf{\underline{3}} \\   \textbf{\underline{1}} & 3 & 2  \\  2 & 3 & 1 \end{array}
$$$$
\begin{array} {ccc} \textbf{\underline{1}} & 2 & 3  \\  1 & \textbf{\underline{2}}& 3 \\   1 & 3 & 2  \\  2 & 1 & \textbf{\underline{3}} \end{array}~~~~~~
\begin{array} {ccc} 1 & 2 & 3  \\  1 & 2 & 3 \\   2 & 3 & 1  \\  2 & 3 & 1 \end{array}~~~~~~
\begin{array} {ccc} \textbf{\underline{1}} & 2 & 3  \\  1 & 2 & 3 \\   2 & \textbf{\underline{3}} & 1  \\  3 & 1 & \textbf{\underline{2}} \end{array}~~~~~~
\begin{array} {ccc} \textbf{\underline{1}} & 2 & 3  \\  1 & \textbf{\underline{3}} & 2 \\   2 & 1 & 3  \\  3 & 1 & \textbf{\underline{2}} \end{array}~~~~~~
\begin{array} {ccc} 1 & \textbf{\underline{2}} & 3  \\  \textbf{\underline{1}} & 3 & 2 \\   2 & 1 & \textbf{\underline{3}}  \\  3 & 2 & 1 \end{array}
$$
For each row-latin rectangle, except for the rectangle $T$, a transversal is underlined. The second property of the rectangle $T$ is verified directly.
\end{proof}

\begin{lem} \label{partdiag}
In a doubly stochastic matrix of order $4$ every positive partial diagonal of length $2$ can be extended to a positive partial diagonal of length $3$.
\end{lem}

\begin{proof}
Assume that $A$ is a doubly stochastic matrix of order $4$  and $p$ is a positive partial diagonal of length $2$ that cannot be extended to a positive partial diagonal of length $3$. Then all positive entries of the matrix $A$ share a row or a column with at least one of elements of $p$. Equivalently, all positive entries of $A$  can be covered by exactly two rows and columns. Since the sum of entries in each row and each column is exactly $1$, we have that in the intersection of these rows and columns all entries are zero: a contradiction with positivity of $p$. 
\end{proof}

\section{Proof of Theorem~\ref{poly44}}

\begin{proof}
Let us try to construct  a 4-dimensional polystochastic matrix $A$ of order 4 with a zero permanent. The construction takes four steps. 

\textbf{Step 1.} Without loss of generality, assume that entry $a_{0,0,0,0}$ is greater than zero.

Consider the $2$-dimensional plane $B$ composed of indices of the form $(*,*,0,0)$, where $*$ means arbitrary symbol from $\left\{0, \ldots, 3\right\}$. The matrix $B$ is doubly stochastic, so by the Birkhoff theorem, it contains a positive diagonal. Without loss of generality, let entries of the matrix $A$ with indices $(i,i,0,0)$, $i \in \left\{0, \ldots, 3\right\}$ be positive.

\textbf{Step 2.} Let us denote by $B_i$ the 2-dimensional planes of $A$ composed of indices $(i,i,*,*)$. As before, each $B_i$ is a doubly stochastic matrix. Assume that $p_i = \left\{(i,i,\beta_i^j,\gamma_i^j)\right\}_{j=1}^4$ is a  positive diagonal in the matrix $B_i$ containing index $(i,i,0,0)$. Consider the $4 \times 3$ rectangle $R$ for which an entry in a $(i+1)$-th row and in a $\beta_i^j$-th column is equal to $\gamma_i^j$. It is not hard to observe that $R$ is a row-latin rectangle and that each transversal in $R$ gives a positive diagonal in the matrix $A$.

By Lemma~\ref{rectangle}, rectangle $T$ is the unique up to equivalence $4 \times 3$ row-latin rectangle with no transversals. Moreover, changing any symbol of $T$ produces a transversal. So we may assume that entries of the matrix $A$ with the following indices obtained from the rectangle $T$
\begin{gather*}
(0,0,0,0),~(0,0,1,1),~(0,0,2,2),~(0,0,3,3), \\
(1,1,0,0),~(1,1,1,1),~(1,1,2,2),~(1,1,3,3), \\
(2,2,0,0),~(2,2,1,2),~(2,2,2,3),~(2,2,3,1), \\
(3,3,0,0),~(3,3,1,2),~(3,3,2,3),~(3,3,3,1)
\end{gather*}
 are positive and that for all other indices of the form $(i,i,\beta, \gamma)$, where $i \in \left\{0, \ldots, 3\right\}$ and $\beta, \gamma \in \left\{1, 2, 3\right\}$, the entries of $A$ are equal to zero.

\textbf{Step 3.} For $k  \in \left\{1,2,3\right\}$ denote by $C_k$ the $2$-dimensional planes of $A$ composed of indices $(*,*,k,k)$. Note that the doubly stochastic matrices $C_k$ contain positive partial diagonals of length 2 formed by indices $(0,0,k,k)$ and $(1,1,k,k)$. By Lemma~\ref{partdiag},  each of these diagonals can be extended to a positive partial diagonal of length 3 by new indices $(\mu_k, \nu_k,k,k)$, where $\mu_k, \nu_k \in \left\{2,3\right\}$ and $\mu_k \neq \nu_k$. 

If for some $k_1, k_2$ it holds $\mu_{k_1} = \nu_{k_2} = 2$ and $\mu_{k_2} = \nu_{k_1} = 3$ then we have a positive diagonal 
$$\left\{(0,0,0,0), (1,1, k_3, k_3), (\mu_{k_1}, \nu_{k_1}, k_1, k_1), (\mu_{k_2}, \nu_{k_2}, k_2, k_2)\right\}$$
in the matrix $A$, where $k_3 \neq k_2, k_1$.

Therefore, the last remaining possibility for $A$ do not have a positive diagonal  is that for each $k \in \left\{1,2,3\right\}$ all entries with indices $(2,3, k, k)$ are positive and all entries with indices $(3,2, k, k)$ are zero (or vice versa). 

\textbf{Step 4.} For each $k \in \left\{1,2,3\right\}$ consider the line composed of indices of the form $(*,2,k,k)$. Note that this line contains two zero entries, namely entries with indices $(2,2,k,k)$ and $(3,2,k,k)$. If we suppose that an entry with index $(1,2,k,k)$ is equal to zero too, we obtain a contradiction with polystochaticity of the matrix $A$, because in this case we have  $a_{0,2,k,k} = 1$ and $a_{0,0,k,k} >0$. Therefore, all entries $a_{1,2,k,k}$ are greater than zero. By similar reasoning, we have that for each $k \in \left\{1,2,3\right\}$ entries $a_{3,1,k,k}$ are positive. But then the matrix $A$ has a positive diagonal, for example:
$$\left\{ (0,0,0,0), (1,2,1,1), (2,3,2,2), (3,1,3,3) \right\}.$$ 

$$\begin{array} {cccc|cccc|cccc|cccc}
+_1 & . & . & .   &   . & . & . & .   &   . & . & . & .   &   . & . & . & . \\
. & +_2 & 0_2 & 0_2   &   . & . & . & .   &   . & . & . & .   &   . & . & . & . \\
. & 0_2 & +_2 & 0_2   &   . & . & . & .   &   . & . & . & .   &   . & . & . & . \\
. & 0_2 & 0_2 & +_2   &   . & . & . & .   &   . & . & . & .   &   . & . & . & . \\
\hline
. & . & . & .   &   +_1 & . & . & .   &   . & . & . & .   &   . & . & . & . \\
. & . & . & .   &   . & +_2 & 0_2 & 0_2   &   . & +_4 & . & .   &   . & . & . & . \\
. & . & . & .   &   . & 0_2 & +_2 & 0_2   &   . & . & +_4 & .   &   . & . & . & . \\
. & . & . & .   &   . & 0_2 & 0_2 & +_2   &   . & . & . & +_4   &   . & . & . & . \\
\hline
. & . & . & .   &   . & . & . & .   &   +_1 & . & . & .   &   . & . & . & . \\
. & . & . & .   &   . & . & . & .   &   . & 0_2 & +_2 & 0_2   &   . & +_3 & . & . \\
. & . & . & .   &   . & . & . & .   &   . & 0_2 & 0_2 & +_2   &   . & . & +_3 & . \\
. & . & . & .   &   . & . & . & .   &   . & +_2 & 0_2 & 0_2   &   . & . & . & +_3 \\
\hline
. & . & . & .   &   . & . & . & .   &   . & . & . & .   &   +_1 & . & . & . \\
. & . & . & .   &   . & +_4 & . & .   &   . & 0_3 & . & .   &   . & 0_2 & +_2 & 0_2 \\
. & . & . & .   &   . & . & +_4 & .   &   . & . & 0_3 & .   &   . & 0_2 & 0_2 & +_2 \\
. & . & . & .   &   . & . & . & +_4   &   . & . & . & 0_3   &   . & +_2 & 0_2 & 0_2 \\
\end{array}$$

\begin{center}
\textbf{Table 2.} The 4-dimensional matrix $A$ of order 4 after the last step.

``$+$'' denotes a positive entry, ``$0$'' is a zero entry, dots are used for insignificant entries. Subscripts indicate steps on which entries are specified. 
\end{center}

\end{proof}

\end{document}